\documentclass[12pt]{amsart}
\usepackage[dvipsnames]{xcolor}
\usepackage{float, comment, ytableau, hyperref, enumitem, pgffor, circledsteps, etoolbox, xstring, mathtools, amsmath, amssymb, tikz, tikz-cd, quiver, adjustbox}
\usepackage[maxbibnames=99,giveninits,maxnames=10,backend=biber,block=space, style = numeric, sorting = none]{biblatex}

\usepackage[capitalize, nameinlink]{cleveref}

\hypersetup{colorlinks, linkcolor = RoyalPurple, citecolor = RoyalPurple, urlcolor = RoyalPurple}

\title[Even Dimensional Partitions Modulo 4]{Enumeration of even dimensional partitions modulo 4}
\author{Aditya Khanna}
\date{}

\bibliography{refs_even}

\newtheorem{thm}{Theorem}
\newtheorem{lem}[thm]{Lemma}
\newtheorem{prop}[thm]{Proposition}
\newtheorem{cor}[thm]{Corollary}
\theoremstyle{definition}
\newtheorem{defn}[thm]{Definition}
\newtheorem{example}[thm]{Example}
\newtheorem{remark}[thm]{Remark}
\newtheorem{notation}[thm]{Notation}

\newcommand{\nth}[1]{
\IfEqCase*{#1}{%
{1}{$#1^\mathrm{st}$~}%
{2}{$#1^\mathrm{nd}$}%
{3}{$#1^\mathrm{rd}$}%
}[$#1^\mathrm{th}$]
}

\newcommand{\boks}[1]{\ytableausetup{boxsize = #1cm}}

\DeclareMathOperator{\sh}{dg}

\DeclareMathOperator{\core}{{core}}

\DeclareMathOperator{\T}{\mathcal{T}}
\DeclareMathOperator{\bin}{bin}
\ytableausetup{aligntableaux = center, nobaseline}
\allowdisplaybreaks
\begin{document}
\begin{abstract}
The number of standard Young tableaux possible of shape corresponding to a partition $\lambda$ is called the dimension of the partition and is denoted by $f^{\lambda}$. Partitions with odd dimensions were enumerated by McKay and were further characterized by Macdonald using the theory of 2-core towers. We use the same theory to extend the results to partitions of $n$ with dimensions congruent to 2 modulo 4 which are enumerated by $a_2(n)$. We provide explicit results for $a_2(n)$ when $n$ has no consecutive 1s in its binary expansion and give a recursive formula to compute $a_2(n)$ for all $n$.
\end{abstract}
\maketitle

\section{Introduction}
\subsection{Background} 
We call $\lambda = (\lambda_1, \ldots, \lambda_k)$ a \textit{partition of $n$} if the entries are weakly decreasing and sum to $n$. The \textit{dimension of the partition $\lambda$}, denoted by $f^\lambda$, is the number of standard Young tableaux (\cref{introsec}) on shape $\lambda$. Denote by $m_p(n)$, the number of partitions of $n$ with dimension not divisible by $p\in \mathbb{N}$.\\
McKay \cite{mckay} computed $m_2(n)$ and Macdonald \cite{mcd} gave a complete answer for $m_p(n)$ using the theory of $p$-core towers. For $n = 2^{k_1}+\ldots + 2^{k_\ell}$ with $k_1 > \ldots > k_\ell$, we have $m_2(n) = 2^{k_1 + \ldots + k_\ell}$.
In particular, $m_2(n)$ enumerates partitions of $n$ whose dimensions are odd. We call these \textit{odd partitions} and will denote their count by $a(n)$.

Recently, there has been some interest in extending Macdonald's results. In their paper, Amrutha P and T. Geetha \cite{geetha} tackle the problem of computing $m_{2^k}(n)$. Their equation (6) gives a general solution to this problem but the results are not amenable to enumeration.
They provide explicit results for $m_4(n)$ when $n = 2^\ell$ which follows from \cref{2mod4cor}. In this paper, we provide explicit recursions which can be used to compute $m_4(n)$ for a general $n$ (\cref{2mod4thm}) and find a special case that lends itself to a pleasing closed form (\cref{2mod4cor}).
We prove this via the theory of 2-core towers which has been used to enumerate chiral partitions \cite{spallone1} and other representations of Coxeter groups \cite{spallone2}. The main motivation for computing values modulo 4 comes from the classification of spinorial representations of the symmetric group \cite{jyo}.

\subsection{Main Results}\label{ssec:mainresults}
Let $a_i(n)$ denote the number of partitions of $n$ with dimension congruent to $i$ modulo 4. Then the number of odd partitions is $a(n)= m_2(n) = a_1(n) + a_3(n)$. 

\begin{thm}\label{2mod4thm}
Let $n = 2^R + m$ such that $m < 2^R$. Then we have 
\[a_2(n) = \begin{cases}
2^R\cdot a_2(m) + \binom{2^{R-1}}{2}\cdot a(m), & \text{ if }m < 2^{R-1}\\
2^R\cdot a_2(m) + \left(\binom{2^{R-1}}{3} + 2^{R-1}\right) \cdot \displaystyle{\frac{a(m)}{2^{R-1}}}, & \text{ if } 2^{R-1} \leq m < 2^R.
\end{cases}\]
\end{thm}
We call a positive integer $n$ \textit{sparse}\footnote{These numbers are called \textit{fibbinary numbers} in literature as they are enumerated by the Fibonacci numbers. They are the OEIS entry\tt{ A003714}.} if it does not have any consecutive ones in its binary expansion, or equivalently for $n = 2^{k_1}+\ldots + 2^{k_\ell}$, we have $k_i > k_{i+1} + 1$ . For instance 42 is sparse as it has the binary expansion $2^5 + 2^3 + 2^1$. Define $\nu(n) := \ell$, so $\nu(42) = 3$.
\begin{cor}\label{2mod4cor}
When $n$ is sparse, we have 
\[a_2(n) = \begin{cases}
\displaystyle{\frac{a(n)}{8}(n - 2 \nu(n))},& \text{if }n \text{ is even}\\
a_2(n-1), &\text{if } n \text{ is odd}.
\end{cases}\]
\end{cor}

\section{Notations and Definitions}\label{sec:notation}
\subsection{Partitions, Ferrers Diagrams and Young Tableaux}\label{introsec}
The definitions and proofs in this section are drawn heavily from the monograph of Olsson \cite{olsson}.
\begin{defn}
We call $\lambda = (\lambda_1, \ldots, \lambda_k)$ a \textit{partition} of \textit{size} $n$ if $\lambda_i$ is a positive integer for all $i$ and $|\lambda|:=\sum\limits_{i=1}^k \lambda_i = n$. We write $\lambda\vdash n$.
\end{defn}

We can visualize a partition $\lambda$ by constructing a top-left justified array of boxes with the \nth{i} row containing $\lambda_i$ boxes. This is known as the \textit{Ferrers diagram of $\lambda$} and is denoted by $\sh(\lambda)$.
\begin{example}
For  $\lambda = (4, 3, 3, 1)$, we have \boks{0.3}
\begin{center}
$\sh(\lambda)$ = \ydiagram{4, 3, 3, 1}
\end{center}\end{example} 
\ytableausetup{boxsize = normal}
For $\lambda\vdash n$, we can fill the boxes of $\sh(\lambda)$ with numbers in the set $\{1, 2, \ldots, n\}$. A \textit{standard Young tableau (SYT)} on the shape $\sh(\lambda)$ is a filling such that the entries increase strictly along rows (left to right) and along columns (top to bottom). This implies that each integer from 1 through $n$ appears exactly once in an SYT.
\begin{example}
Continuing with the above example, the following filling is an SYT on $\sh((4, 3, 3, 1))$.
\boks{0.42}
\begin{center}
\ytableaushort{1 2 7 8, 3 5 {10}, 4 9 {11}, 6}
\end{center}

\end{example}

The number of SYTs on $\sh(\lambda)$ is called the \textit{dimension of $\lambda$} and is denoted by $f^{\lambda}$. We call a partition \textit{odd} if $f^\lambda$ is odd and \textit{even} otherwise.

\subsection{Hooks and hook-lengths}
We call the boxes in Ferrers diagrams \textit{cells}. We label the cell in the \nth{i} row from top and \nth{j} column from left by $(i,j)$. For $\lambda = (\lambda_1, \ldots, \lambda_k)$, write $(i,j) \in \sh(\lambda)$ if $1\leq i\leq k$ and $1\leq j \leq \lambda_i$. We say $\sh(\lambda)$ contains a \textit{removable domino} or a \textit{removable 2-hook} if the pair $(i,j), (i,j+1)\in \sh(\lambda)$ or $(i,j), (i+1,j)\in \sh(\lambda)$ can be removed and the rows still weakly decrease in size in the new shape. The condition of being able to remove $(i,j), (i,j+1)\in \sh(\lambda)$ is equivalent to $\lambda$ containing a pair of parts such that $\lambda_i \geq \lambda_{i+1} + 2$. The other condition of being able to remove $(i,j), (i+1,j)\in \sh(\lambda)$ is equivalent to $\lambda$ containing more than one part of size $j$. A partition $\lambda$ that does not contain a removable domino is called a \textit{2-core.}
\begin{lem}\label{lem:2cores}
A partition $\lambda$ is a 2-core if and only if $\lambda = (n , n-1, \ldots 3,2,1)$ for some $n\geq 0$.
\end{lem}
\begin{proof}
    A partition which has no consecutive parts that differ by more than 1 in size and contains no repeated parts must have all parts which are consecutive integers ending in 1. The converse is easy to check.
\end{proof}
It is well-known \cite[Prop 3.2]{olsson} that successive removal of dominoes from any $\sh(\lambda)$ results in a 2-core and this 2-core is independent of the order in which dominoes are removed. We call the 2-core thus obtained the \textit{2-core of $\lambda$} and denote it by $\core_2(\lambda)$.
\begin{example}
    The partition $(4,3,3,1)$ has the 2-core $(2,1)$ which can be obtained as shown
    \boks{0.3}
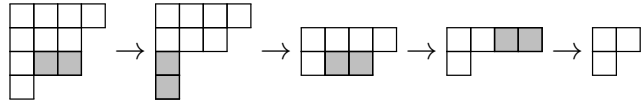
    \[
    \ydiagram{4,3,3,1}*[*(lightgray)]{0,0,1+2} \to \ydiagram{4,3,1,1}*[*(lightgray)]{0,0,1,1} \to \ydiagram{4,3}*[*(lightgray)]{0,1+2} \to \ydiagram{4,1}*[*(lightgray)]{2+2} \to \ydiagram{2,1}
    \]
\end{example}
In the literature, an analog of the division algorithm is defined for partitions using the concept of $2$-cores as a remainder and defining \textit{$2$-quotients} which are pairs of partitions $(\lambda^{(0)}, \lambda^{(1)})$ satisfying the analog of Euclid's division lemma \cite[Prop 3.6, (ii)]{olsson}
\[
|\lambda| = 2\left(|\lambda^{(0)}| + |\lambda^{(1)}|\right) + |\core_2(\lambda)|.
\]
To construct a 2-quotient of $\lambda$, we first label each cell $(i,j)\in \sh(\lambda)$ with $0$ if $i+j$ is even and $1$ if $i+j$ is odd. Define $\lambda^{(0)}$ to be the partition formed by the subset of cells $(i,j)$ such that the rightmost entry of row $i$ is 0 and the bottom entry of column $j$ is 1. Similarly, define $\lambda^{(1)}$ to be the partition defined by the subset of cells $(i,j)$ such that rightmost entry of row $i$ is 1 and the bottom entry of column $j$ is 0.
\begin{example}\boks{0.4}
    The partition $(4,3,3,1)$ can be filled as 
\[\ytableaushort{0101,101,010,1} \]
We have $\lambda^{(0)} = (2)$ which is formed by the cells $(3,1)$ and $(3,2)$. This is because the third row ends in a zero while both the first and the second column end in a 1. We also find $\lambda^{(1)} =(1,1)$ formed by the cells $(1,3)$ and $(2,3)$.
\end{example}
It is possible to reconstruct the original partition given the 2-core and 2-quotient \cite[Prop 3.7]{olsson}.
With these notions in hand, we are ready to define \textit{2-core towers}. But first, we need some notation for clarity.
\begin{notation}
Let $n\geq 0$. We write $\lambda^{(ij)}$ for $(\lambda^{(i)})^{(j)}$ for any binary string $i\in \{0,1\}^n$ and $j\in \{0,1\}$. Also, $\lambda^{(\varnothing)} = \lambda$ and $(\lambda^{(\varnothing)})^{(j)} = \lambda^{(j)}$.
\end{notation}
We recall the construction of 2-core towers as presented in \cite{olsson}. For a partition $\lambda$, construct an infinite rooted binary tree with nodes labeled by 2-cores as follows:
\begin{itemize}
\item Label the root node with $\core_2(\lambda) = \core_2(\lambda^{(\varnothing)})$.
\item If the length of the unique path from the root node to a node $v$ is $i$, then we say that the \textit{node $v$ is in the $i^{\mathrm{th}}$ row.} 
\item Every node in the $i^{\mathrm{th}}$ row has edges leading down to two nodes in the $(i+1)^{\text{st}}$ row. Define a recursive labeling as follows: if the label of some node $v$ in the $i^{\mathrm{th}}$ row is $\core_2(\lambda^{(b)})$ for some binary string $b$, then the two nodes in the $(i+1)^{\text{st}}$ row with parent node $v$ have labels $\core_2(\lambda^{(b0)})$ and $\core_2(\lambda^{(b1)})$ respectively.
\end{itemize}
This tree is known as the\textit{ 2-core tower of $\lambda$.}
\begin{example}
We start with the partition $(6,5,4,2,1,1)$. In order to compute its 2-core tower, we first compute the $2$-quotients at each stage. We draw a tree such that for a node labeled by $\lambda$, the left child is labeled by $\lambda^{(0)}$ and the right child is labeled by $\lambda^{(1)}$.\\
\adjustbox{scale = {0.5}{0.65}, center}
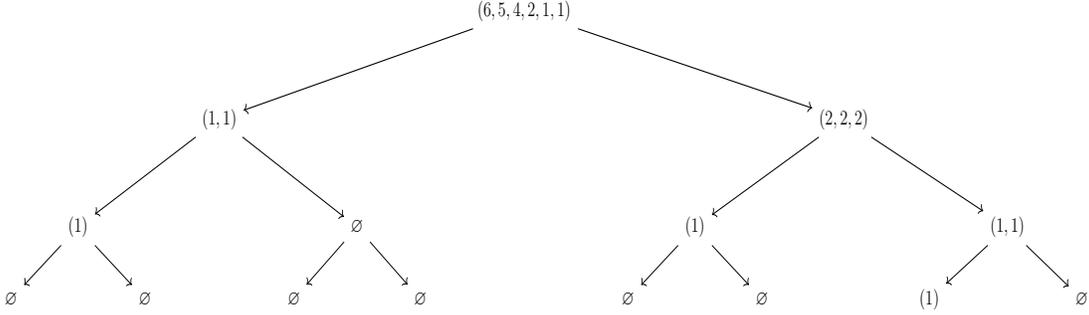
{%
\begin{tikzcd}
	&&&&&&& {(6,5,4,2,1,1)} \\
	\\
	&&& {(1,1)} &&&&&&&& {(2,2,2)} \\
	\\
	& {(1)} &&&& \varnothing &&&& {(1)} &&&& {(1,1)} \\
	\varnothing && \varnothing && \varnothing && \varnothing && \varnothing && \varnothing && {(1)} && \varnothing
	\arrow[from=5-2, to=6-1]
	\arrow[from=5-2, to=6-3]
	\arrow[from=5-6, to=6-5]
	\arrow[from=5-6, to=6-7]
	\arrow[from=5-10, to=6-9]
	\arrow[from=5-10, to=6-11]
	\arrow[from=5-14, to=6-13]
	\arrow[from=3-4, to=5-2]
	\arrow[from=3-4, to=5-6]
	\arrow[from=5-14, to=6-15]
	\arrow[from=1-8, to=3-4]
	\arrow[from=1-8, to=3-12]
	\arrow[from=3-12, to=5-14]
	\arrow[from=3-12, to=5-10]
\end{tikzcd}
}
The corresponding 2-core tower is given by taking the 2-core of every node.\\

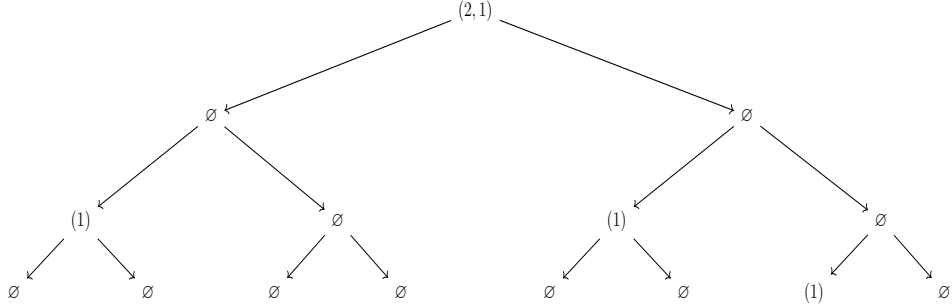
\adjustbox{scale = {0.5}{0.65}, center}{%
\begin{tikzcd}
	&&&&&&& {(2,1)} \\
	\\
	&&& \varnothing &&&&&&&& \varnothing \\
	\\
	& {(1)} &&&& \varnothing &&&& {(1)} &&&& \varnothing\\
	\varnothing && \varnothing && \varnothing && \varnothing && \varnothing && \varnothing && {(1)} && \varnothing
	\arrow[from=5-2, to=6-1]
	\arrow[from=5-2, to=6-3]
	\arrow[from=5-6, to=6-5]
	\arrow[from=5-6, to=6-7]
	\arrow[from=5-10, to=6-9]
	\arrow[from=5-10, to=6-11]
	\arrow[from=5-14, to=6-13]
	\arrow[from=3-4, to=5-2]
	\arrow[from=3-4, to=5-6]
	\arrow[from=5-14, to=6-15]
	\arrow[from=1-8, to=3-4]
	\arrow[from=1-8, to=3-12]
	\arrow[from=3-12, to=5-14]
	\arrow[from=3-12, to=5-10]
\end{tikzcd}
}
All the nodes in the subsequent rows are $\varnothing$, and so we omit them.
\end{example}

Every partition has a unique 2-core tower and every 2-core tower comes from a unique partition. This bijection is inherited from the core-quotient bijection. The constraints on the dimension of our partition also leads to constraints on the node labels of its 2-core tower. In particular, we will find restrictions on the sizes of partitions that can appear in the $k$th row.
\begin{defn}\label{defn:weight}
For $k\geq 0$, the \textit{sum of sizes of partitions in the $k^{\text{th}}$ row of the 2-core tower of $\lambda$ }is given by
\[
w_k(\lambda) := \sum\limits_{b\in\{0,1\}^k} |\core_2(\lambda^{(b)})|.
\]
\end{defn}

For all of our subsequent discussion, we will let $n = \sum_{i\geq 0} b_i2^i$, where $b_i\in\{0,1\}$. Define $\bin(n) := \{i\mid b_i = 1\}$ and $\bin'(n) := \bin(n)\backslash\{0\}$. Macdonald~\cite{mcd} gave a characterization of odd partitions $\lambda$ in terms of 2-core towers using the binary expansion of $|\lambda|$.
\begin{prop}\label{prop:mcdmain1}
Let $n$ be as above, then
    $\lambda$ is an odd partition if and only $w_i(\lambda) = b_i$ for all $i\geq 0$.
\end{prop}
\begin{proof}
     The proof can be found in Section 4 of \cite{mcd}. 
\end{proof}

\section{Computing $a_2(n)$}\label{sec:2mod4}
Similar to the above proposition, we can find a simple description of partitions with dimension congruent to 2 modulo 4 in terms of 2-core towers. Define $v_2(n)$ to be the power of 2 that appears in the prime factorization of $n$. 
\begin{prop}\label{prop:mcdmain}
Let $\lambda$ be a partition of $n$. Then,
$v_2(f^{\lambda}) = 1$ if and only if for some $R \in \bin'(n)$, we have $w_{R -1}(\lambda) = b_{R-1} + 2$, $w_{R}(\lambda) = 0$ and $w_i(\lambda) = b_i$ for all $i\neq R, R-1$.
\end{prop}
\begin{proof}
We write $w_i:= w_i(\lambda)$.
Through the exposition in section 4 of \cite{mcd}, we know that there exists a sequence of non-negative integers $(z_i)_{i\geq 0}$ with $z_0 = 0$ such that $
w_i + z_i = b_i + 2z_{i+1}$.\footnote{The notation in \cite{mcd} differs from ours as Macdonald uses $a_i$ instead of $w_i$, $\alpha_i$ instead of $b_i$ and $b_i$ instead of $z_i$.}
Summing over all non-negative numbers gives
\[
\sum\limits_{i\geq 0} w_i + \sum\limits_{i\geq 0} z_i = \sum\limits_{i\geq 0} b_i + 2\sum\limits_{i\geq 1} z_i.
\]
From Section 3 of \cite[Eqs. (3.3), (3.4)]{mcd}, we find that
\[
\sum\limits_{i\geq 0} w_i - \sum\limits_{i\geq 0}b_i  = v_2(f^\lambda)
\]
which reduces in our case to $\sum\limits_{i\geq 0} w_i = \sum\limits_{i\geq 0}b_i  + 1$. Using this and rearranging the sums gives 
\[
z_0 + 1 = \sum\limits_{i\geq 1} z_i
\]
As $z_0 = 0$, we can deduce that $z_i = 1$ for exactly one $i > 0$ and zero for the rest. Putting $z_R = 1$ gives us the equations,
\begin{align*}
w_{R-1} &= b_{R-1} + 2\\
w_R + 1 &= b_R.
\end{align*}
Notice that the second equation forces $b_R = 1$ and so $z_R = 1$ is possible only when $R \in \bin'(n)$. We find $w_R = 0$ and $w_{R-1} = b_{R-1} + 2$. The rest is clear as $z_i = 0$ for $i\neq R$.
\end{proof}
This proposition places explicit restrictions on the sizes of partitions possible in each row of the 2-core tower. For odd partitions, \cref{prop:mcdmain1} tells us that each row $i$ can be either empty or have one box depending on the bit $b_i$. This phenomenon carries over to this case except for rows $R$ and $R-1$ wherein the former is empty, while the latter contains either 2 or 3 boxes. The computation of $a_2(n)$ will follow from the combinatorics of how these boxes are distributed.

\begin{notation}
For integers $w,k\geq 0$ and $\mu_{[i]}$ a 2-core for $1\leq i \leq 2^k$, let $\T_k(w)$ denote the number of solutions $\left(\mu_{[1]}, \ldots, \mu_{[2^k]}\right)$ to \[\sum\limits_{i=1}^{2^k}|\mu_{[i]}| = w.\]
\end{notation}
\begin{lem}\label{lem:weight}
With the above notation, we have
\[
\T_k(w) = \begin{cases}
1, & \text{if }w = 0\\
2^k, &\text{if } w = 1\\
\binom{2^k}{2},&\text{if } w = 2\\
\binom{2^k}{3} + 2^k, & \text{if }w = 3.
\end{cases}
\]
\end{lem}
\begin{proof}
We prove the lemma by dividing it into cases and using the result in \cref{lem:2cores}.
\begin{enumerate}
\item When $w = 0$, the only solution is $\mu_{[i]}= \varnothing$ for all $i$. Thus, $\T_k(0) = 1$. 
\item When $w = 1$, we have $2^k$ possible options for $\mu_{[i]} = (1)$ as $1\leq i\leq 2^k$ which gives $\T_k(1) = 2^k$.
\item For $w = 2$, we must choose $\mu_{[i]} = \mu_{[j]} = (1)$, for some $i$ and $j$, and the rest $\varnothing$. Thus, $\T_{k}(2) = \binom{2^k}{2}$.
\item In this case, we can either choose $i_1, i_2, i_3$ such that $\mu_{[i_1]} = \mu_{[i_2]} = \mu_{[i_3]} = (1)$ which can be done in $\binom{2^k}{3}$ ways or choose a $j$ such that $\mu_{[j]} = (2,1)$ which can be done in $2^k$ ways. This gives us $\T_k(3) = \binom{2^k}{3} + 2^k$.
\end{enumerate}
\end{proof}

\begin{notation}\label{notation:wk}
Let $n =\sum\limits_{i\geq 0} b_i2^i$ with $b_i\in \{0,1\}$. For $k\in\bin'(n)$, define a sequence of non-negative integers, $\mathbf{w}_k(n) := (w^k_i(n))_{i\geq 0}$ with the following properties:
\begin{enumerate}
\item $w^k_{k-1}(n) = b_{k-1}+2$,
\item $w^k_{k}(n) = 0$, and 
\item $w^k_i(n) =  b_i$ for all other values of $i$.
\end{enumerate}
Let $\T(\mathbf{w}_k(n)) := \prod\limits_{i\geq 0} \T_i(w^k_i(n))$. 
\end{notation}
A sequence $\mathbf{w}_k$ corresponds to constructing a 2-core tower of $\lambda$ with $v_2(f^\lambda)= 1$ such that $w_i(\lambda) = w^k_i(n)$, for some $k\in \bin'(n)$. The quantity $\T(\mathbf{w}_k(n))$ counts all such partitions.

\begin{proof}[Proof of \cref{2mod4thm}]
Fix an $n\geq 1$. By the above discussion as well as \cref{prop:mcdmain}, we have that the number of partitions with dimensions congruent to 2 modulo 4 is given by
\begin{align*}
a_2(n) &= \sum\limits_{k\in \bin'(n)} \T(\mathbf{w}_k(n))\\
&= \sum\limits_{k\in \bin'(n)} \prod\limits_{i\geq 0} \T_i(w^k_i(n)).
\end{align*}
Denote the binary expansion of $n$ by $\sum\limits_{i\geq 0}b_i2^i$ as before. If we suppose $n = 2^R + m$ with $m<2^R$, then $b_R = 1$ and $b_i = 0$ for $i > R$. For $k\in \bin'(m) = \bin'(n)\backslash\{R\}$, we have $w^k_i(n) = w^k_i(m)$ for all $i\neq R$ (*). We break the sum up:
\[
a_2(n) = \left(\sum\limits_{k\in \bin'(n)\backslash \{R\}} \prod\limits_{i\geq 0} \T_i(w^k_i(n)) \right)+ \T(\mathbf{w}_R(n)).
\]
For the first term, we have
\begin{align*}
\sum\limits_{k\in \bin'(n)\backslash \{R\}} \prod\limits_{i\geq 0} \T_i(w^k_i(n)) &= \sum\limits_{k\in \bin'(n)\backslash \{R\}} \T_R(w^k_R(n))\prod\limits_{i \neq R} \T_i(w^k_i(n)) \\
&= 2^R\sum\limits_{k\in \bin'(m)} \prod\limits_{i\neq R} \T_i(w^k_i(n))\qquad{\color{lightgray}\text{(as } w_R^k(n) = b_R = 1)}\\
&= \frac{2^R}{\T_R(0)}\sum\limits_{k\in \bin'(m)} \T_R(0)\prod\limits_{i\neq R} \T_i(w^k_i(m)) \quad{\color{lightgray}\text{(by (*))}}\\
&= \frac{2^R}{\T_R(0)}\sum\limits_{k\in \bin'(m)} \prod\limits_{i\geq 0} \T_i(w^k_i(m))\\
&= 2^R a_2(m).
\end{align*}
Now, we deal with the second term:
\begin{align*}
\T(\mathbf{w}_R(n)) &= \T_{R-1}(w^R_{R-1}(n))\T_{R}(w^R_R(n))\prod\limits_{i\neq R, R-1} \T_i(w^R_i(n))\\
&= \T_{R-1}(w^R_{R-1}(n))\T_{R}(w^R_R(n))\prod\limits_{i\neq R, R-1} \T_i(b_i)\\
&= \frac{\T_{R-1}(w^R_{R-1}(n))\T_{R}(w^R_R(n))}{\T_{R-1}(b_{R-1})\T_{R}(b_R)} a(n).
\end{align*}

The appearance of $a(n)$ is the consequence of \cref{prop:mcdmain1} and the cases $w = 0,1$ in \cref{lem:weight}. We have $b_R = 1$, $w^R_R(n) = 0$ and $w^R_{R-1}(n) = b_{R-1} + 2$. This simplifies the above expression to
\[
\frac{\T_{R-1}(b_{R-1} + 2)}{2^{R}\T_{R-1}(b_{R-1})} a(n).
\]
By \cref{lem:weight} and $a(n) = 2^Ra(m)$, we can complete the proof to get 
\[a_2(n) = \begin{cases}
2^R\cdot a_2(m) + \binom{2^{R-1}}{2}\cdot a(m), & \text{ if }m < 2^{R-1}\\
2^R\cdot a_2(m) + \left(\binom{2^{R-1}}{3} + 2^{R-1}\right) \cdot \displaystyle{\frac{a(m)}{2^{R-1}}}, & \text{ if } 2^{R-1} < m < 2^R.
\end{cases}\] 
The first case corresponds to $b_{R-1} = 0$ and the second one corresponds to $b_{R-1} = 1$.
\end{proof}

Recall that a \textit{sparse} number has no consecutive 1s in its binary expansion. In that case, this theorem takes a closed form as evident in \cref{2mod4cor}. Although we can show the corollary using the recursive relations in we just proved, we choose to do perform direct computations as we believe they are more illuminating.
\begin{proof}[Proof of \cref{2mod4cor}]
Let $n$ be a sparse number. The summation as before is:
\begin{align*}
a_2(n) &= \sum\limits_{k\in \bin'(n)} \T(\mathbf{w}_k(n))\\
&= \sum\limits_{k\in \bin'(n)} \frac{\T_{k-1}(w^k_{k-1}(n))\T_{k}(w^k_k(n))}{\T_{k-1}(b_{k-1})\T_{k}(b_k)} a(n).
\end{align*}
As $n$ is sparse, for $k\in \bin'(n)$, we have $b_{k-1} = 0$, $w^k_{k}(n) = 0$ and $w^k_{k-1}(n) = b_{k-1} +2 = 2$.\\
This gives us the following summation,
\[
a_2(n) = \sum\limits_{k\in \bin'(n)} \frac{\T_{k-1}(2)}{\T_{k}(1)} a(n).
\]
By using the results from \cref{lem:weight}, this simplifies to
\begin{align*}
a_2(n) &=  a(n)\sum\limits_{k\in \bin'(n)} \frac{\binom{2^{k-1}}{2}}{2^k}\\
&=  a(n)\sum\limits_{k\in \bin'(n)} \frac{(2^{k-1})(2^{k-1}-1)}{2\cdot 2^k}\\
&=  \frac{a(n)}{8}\sum\limits_{k\in \bin'(n)} 2^k - 2.\\
\end{align*}
If $n$ is even, then this summation becomes $n - 2\nu(n)$. If $n$ is odd, the summation becomes $(n-1) + 2(\nu(n) - 1) = (n-1) + 2(\nu(n-1))$. This gives us the final answer as, 
\[a_2(n) = \begin{cases}
\displaystyle{\frac{a(n)}{8}(n - 2 \nu(n))},& \text{if }n \text{ is even}\\
a_2(n-1), &\text{if } n \text{ is odd}.
\end{cases}\]
\end{proof}
\begin{remark}
The theorem above immediately leads us to the value of $m_4(n)$. We have $m_4(n) = a_1(n) + a_2(n) + a_3(n) = a(n) + a_2(n)$ by definition. If the number of partitions of $n$ is denoted by $p(n)$, then the number of partitions with dimensions divisible by 4 is given by $p(n) - a(n) - a_2(n)$.
\end{remark}

\section{Closing Remarks and Open Problems}\label{sec:problems}
This paper deals with the case of 2 modulo 4 which can be generalized. One may wish to compute even dimensional partitions modulo $2^k$ for higher values of $k$ but the recursions have more cases and too many terms to be of practical use. Another generalization might be to compute partitions with dimensions $p$ modulo $p^2$, but this endeavor is also racked with a multitude of cases and cumbersome calculations. The recursions obtained here and the specialization in the sparse case make this problem perfectly situated to solve explicitly.

\section{Acknowledgments}
This work was completed at the Indian Institute of Science Education and Research, Pune under the supervision of Steven Spallone who introduced me to this fascinating problem and provided help and feedback. I thank Jyotirmoy Ganguly for the discussions and encouragement. Lastly, my sincerest thanks to Nick Loehr for his extremely detailed and consistent feedback on the manuscript.
\printbibliography
\end{document}